\documentclass{amsart}

\usepackage{latexsym,enumerate}
\usepackage{amsmath,amsthm,amsopn,amstext,amscd,amsfonts,amssymb}
\usepackage[ansinew]{inputenc}
\usepackage{verbatim}
\usepackage{wasysym}
\usepackage[all]{xy}
\usepackage{epsfig} 
\usepackage{graphicx,psfrag} 
\usepackage{subfigure}
\usepackage{pstricks,pst-node}

\usepackage{multicol}
\usepackage{color}
\usepackage{colortbl}

\newtheorem{theorem}{Theorem}[section]
\newtheorem{lemma}[theorem]{Lemma}
\newtheorem{proposition}[theorem]{Proposition}

\newtheorem{example}[theorem]{Example}
\newtheorem{remark}[theorem]{Remark}

\DeclareMathOperator{\diam}{diam}

\newcommand{\spb}[1]{\smallskip}
\newcommand{\mpb}[1]{\medskip}
\newcommand{\bpb}[1]{\bigskip}

\renewcommand{\a}{\alpha}
\renewcommand{\b}{\beta}
\newcommand{\e}{\varepsilon}
\renewcommand{\d}{\delta}

\newcommand{\g}{\gamma}
\newcommand{\G}{\Gamma}

\begin{document}
\DeclareGraphicsExtensions{.jpg,.pdf,.mps,.png}
\title{GROMOV HYPERBOLICITY OF MINOR GRAPHS}

\author[Walter Carballosa]{Walter Carballosa$^{(1)}$}
\address{National council of science and technology (CONACYT) $\&$ Autonomous University of Zacatecas,
Paseo la Bufa, int. Calzada Solidaridad, 98060 Zacatecas, ZAC, Mexico}
\email{waltercarb@gmail.com}
\thanks{$^{(1)}$  Supported in part by a grant from Ministerio de Econom{\'\i}a y Competitividad (MTM 2013-46374-P), Spain.}

\author[Jos\'e M. Rodr{\'\i}guez]{Jos\'e M. Rodr{\'\i}guez$^{(1)}$}
\address{Department of Mathematics, Universidad Carlos III de Madrid,
Avenida de la Universidad 30, 28911 Legan\'es, Madrid, Spain}
\email{jomaro@math.uc3m.es}

\author[Omar Rosario]{Omar Rosario$^{(1)}$}
\address{Department of Mathematics, Universidad Carlos III de Madrid,
Avenida de la Universidad 30, 28911 Legan\'es, Madrid, Spain}
\email{orosario@math.uc3m.es}

\author[Jos\'e M. Sigarreta]{Jos\'e M. Sigarreta$^{(1)}$}
\address{Faculty of Mathematics, Autonomous University of Guerrero,
Carlos E. Adame No.54 Col. Garita, 39650 Acalpulco Gro., Mexico}
\email{jsmathguerrero@gmail.com}

\date{\today}

\maketitle{}


\begin{abstract}
If $X$ is a geodesic metric space and $x_1,x_2,x_3\in X$, a {\it
geodesic triangle} $T=\{x_1,x_2,x_3\}$ is the union of the three
geodesics $[x_1x_2]$, $[x_2x_3]$ and $[x_3x_1]$ in $X$. The space
$X$ is $\d$-\emph{hyperbolic} $($in the Gromov sense$)$ if any side
of $T$ is contained in a $\d$-neighborhood of the union of the two
other sides, for every geodesic triangle $T$ in $X$.
The study of hyperbolic graphs is an interesting topic since the
hyperbolicity of a geodesic metric space
is equivalent to the hyperbolicity of a graph related to it.
In the context of graphs, to remove and to contract an edge of a graph are natural transformations.
The main aim in this work is to obtain quantitative information about the distortion of the hyperbolicity constant of the graph $G \setminus e$ (respectively, $\,G/e\,$) obtained from the graph $G$ by deleting (respectively, contracting) an arbitrary edge $e$ from it.
This work provides information about the hyperbolicity constant of minor graphs.
\end{abstract}

{\it Keywords:}  Hyperbolic graph; minor graph; edge contraction; deleted edge.

{\it AMS 2010 Subject Classification numbers:}   05C63; 05C75;  05A20.

\section{Introduction}

Hyperbolic spaces play an important role in geometric
group theory and in the geometry of negatively curved
spaces (see \cite{ABCD, GH, G1}).
The concept of Gromov hyperbolicity grasps the essence of negatively curved
spaces like the classical hyperbolic space, Riemannian manifolds of
negative sectional curvature bounded away from $0$, and of discrete spaces like trees
and the Cayley graphs of many finitely generated groups. It is remarkable
that a simple concept leads to such a rich
general theory (see \cite{ABCD, GH, G1}).

The first works on Gromov hyperbolic spaces deal with
finitely generated groups (see \cite{G1}). 
Initially, Gromov spaces were applied to the study of automatic groups in the science of computation
(see, \emph{e.g.}, \cite{O}); indeed, hyperbolic groups are strongly geodesically automatic,
\emph{i.e.}, there is an automatic structure on the group \cite{Cha}.

The concept of hyperbolicity appears also in discrete mathematics, algorithms
and networking. For example, it has been shown empirically
in \cite{ShTa} that the internet topology embeds with better accuracy
into a hyperbolic space than into an Euclidean space
of comparable dimension; the same holds for many complex networks, see \cite{KPKVB}.
A few algorithmic problems in
hyperbolic spaces and hyperbolic graphs have been considered
in recent papers (see \cite{ChEs, Epp, GaLy, Kra}).
Another important
application of these spaces is the study of the spread of viruses through on the
internet (see \cite{K21,K22}).
Furthermore, hyperbolic spaces are useful in secure transmission of information on the
network (see \cite{K27,K21,K22,NS}).
The hyperbolicity has also been applied in the field of random networks. For example, it was shown in \cite{Sha1,Sha2} that several types of small-world networks and networks
with given expected degrees are not hyperbolic in some sense.

The study of Gromov hyperbolic graphs is a subject of increasing interest in graph theory; see, \emph{e.g.},
\cite{BC,BRS,BRSV2,BRST,
CCCR,CPeRS,
CRS,CRSV,CD,
K27,K21,K22,
KPKVB,MRSV,
NS,
PRSV,
PT,R,
RSVV,
T,WZ}
and the references therein.

We say that the curve $\g$ in a metric space $X$ is a
\emph{geodesic} if we have $L(\g|_{[t,s]})=d(\g(t),\g(s))=|t-s|$ for every $s,t\in [a,b]$
(then $\gamma$ is equipped with an arc-length parametrization).
The metric space $X$ is said \emph{geodesic} if for every couple of points in
$X$ there exists a geodesic joining them; we denote by $[xy]$
any geodesic joining $x$ and $y$; this notation is ambiguous, since in general we do not have uniqueness of
geodesics, but it is very convenient.
Consequently, any geodesic metric space is connected.
If the metric space $X$ is
a graph, then the edge joining the vertices $u$ and $v$ will be denoted by $[u,v]$.

In order to consider a graph $G$ as a geodesic metric space, we identify (by an isometry)
any edge $[u,v]\in E(G)$ with the interval $[0,1]$ in the real line;
then the edge $[u,v]$ (considered as a graph with just one edge)
is isometric to the interval $[0,1]$.
Thus, the points in $G$ are the vertices and, also, the points in the interior
of any edge of $G$.
In this way, any connected graph $G$ has a natural distance
defined on its points, induced by taking shortest paths in $G$,
and we can see $G$ as a metric graph.
Throughout this paper, $G=(V,E)$ denotes a connected graph such that every edge has length $1$
and $V\neq \emptyset$.
These properties guarantee that any connected graph is a geodesic metric space.
We will work both with simple and non-simple graphs. The difference between them is that the first type does not contain either loops or multiple edges.
Although the operation of contraction is naturally defined for non-simple graphs, simple graphs are a more usual context in the study of hyperbolicity.

If $X$ is a geodesic metric space and $x_1,x_2,x_3\in X$, the union
of three geodesics $[x_1 x_2]$, $[x_2 x_3]$ and $[x_3 x_1]$ is a
\emph{geodesic triangle} that will be denoted by $T=\{x_1,x_2,x_3\}$
and we will say that $x_1,x_2$ and $x_3$ are the vertices of $T$; it
is usual to write also $T=\{[x_1x_2], [x_2x_3], [x_3x_1]\}$. We say
that $T$ is $\d$-{\it thin} if any side of $T$ is contained in the
$\d$-neighborhood of the union of the two other sides. We denote by
$\d(T)$ the sharp thin constant of $T$, i.e. $ \d(T):=\inf\{\d\ge 0:
\, T \, \text{ is $\d$-thin}\,\}. $ The space $X$ is
$\d$-\emph{hyperbolic} $($or satisfies the {\it Rips condition} with
constant $\d)$ if every geodesic triangle in $X$ is $\d$-thin. We
denote by $\d(X)$ the sharp hyperbolicity constant of $X$, i.e.
$\d(X):=\sup\{\d(T): \, T \, \text{ is a geodesic triangle in
}\,X\,\}.$
We say that $X$ is
\emph{hyperbolic} if $X$ is $\d$-hyperbolic for some $\d \ge 0$; then $X$ is hyperbolic if and only if
$ \d(X)<\infty$.
If we have a triangle with two
identical vertices, we call it a ``bigon". Obviously, every bigon in
a $\d$-hyperbolic space is $\d$-thin.

Trivially, any bounded
metric space $X$ is $(\diam X)$-hyperbolic.
A normed linear space is hyperbolic if and only if it has dimension one.
If a complete Riemannian manifold is simply connected and its sectional curvatures satisfy
$K\leq c$ for some negative constant $c$, then it is hyperbolic.
See the classical references \cite{ABCD,GH} in order to find further results.
We want to remark that the main examples of hyperbolic graphs are the trees.
In fact, the hyperbolicity constant of a geodesic metric space can be viewed as a measure of
how ``tree-like'' the space is, since those spaces $X$ with $\delta(X) = 0$ are precisely the metric trees.
This is an interesting subject since, in
many applications, one finds that the borderline between tractable and intractable
cases may be the tree-like degree of the structure to be dealt with
(see, \emph{e.g.}, \cite{CYY}).

An graph $H$ is a \emph{minor} of a graph $G$ if a graph isomorphic to $H$ can be obtained from $G$ by contracting some edges,
deleting some edges, and deleting some isolated vertices.
Minor graphs is an interesting class of graphs.
This topic started with one well-known result on planar graph, independently proved by  Kuratowski and Wagner, which says that a graph is planar if and only if it do not include as a minor the complete graph $K_5$ nor the complete bipartite graph $K_{3,3}$ (see \cite{Ku,W}).

For a general graph or a general geodesic metric space
deciding whether or not a space is
hyperbolic is usually very difficult.
Therefore, it is interesting to study the invariance of the hyperbolicity of graphs under
appropriate transformations.
The invariance of the hyperbolicity under some natural transformations on graphs have been studied in previous papers, for instance, removing edges of a graph is studied in \cite{BRSV2,CPeRS}, the line graph of a graph in \cite{CRS,CRSV}, the dual of a planar graph in \cite{CPoRS,PRSV} and the complement of a graph in \cite{BRST}.

To remove and to contract edges of a graph are also very natural transformations.
In \cite{CPeRS} the authors study the distortion of the hyperbolicity constant of the graph $G \setminus e$ obtained from a graph $G$ by removing an edge $e$.
These results allow to obtain the characterization, in a quantitative way, of the hyperbolicity of many graphs in terms of local hyperbolicity.
In this work we obtain the invariance of the hyperbolicity under the contraction of a finite number of edges. Besides, we obtain quantitative information about the distortion of the hyperbolicity constant of the graph $G/e$ obtained from
the graph $G$ by contracting an arbitrary edge $e$ from it for simple and non-simple graphs, in Sections \ref{Sect:constract} and \ref{Sect:non-simple}, respectively.
In Sections \ref{Sect:minor} and \ref{Sect:non-simple} we obtain the invariance of the hyperbolicity on many minor graphs as a consequence of these results
for simple and non-simple graphs, respectively.

\section{Hyperbolicity and edge contraction on simple graphs}\label{Sect:constract}

In this section we study the distortion of the hyperbolicity constant by contraction of one edge in any simple graph.
If $G$ is a graph and $e:=[A,B]\in E(G)$, we denote by $G/e$ the graph obtained by contracting the edge $e$ from it
(we remove $e$ from $G$ while simultaneously we merge $A$ and $B$).

\smallskip

Along this work we will denote by $V_e$ the vertex in $G/e$ obtained by identifying $A$ and $B$.

\smallskip

Note that any vertex $v\in V(G)\setminus\{A,B\}$ can be seen as a vertex in $V(G/e)$. Also we can write any edge in $E(G/e)$ in terms of its endpoints, but we write $V_e$ instead of $A$ or $B$.
If $[v,A]$ and $[v,B]$ are edges of $G$ for some $v\in V(G)$, then we replace both edges by a single edge $[v,V_e]\in G/e$ (recall that we work with simple graphs), see Figure \ref{fig:OutlineH}.

We define the map $h: G \rightarrow G/e$ in the following way:
if $x$ belongs to the edge $e$, then $h(x):=V_e$;
if $x\in G$ does not belong to $e$, then $h(x)$ is the ``natural inclusion map".
Clearly $h$ is onto, \emph{i.e.}, $h(G)=G/e$. 
Besides, $h$ is an injective map in the union of edges without endpoints in $\{A,B\}$.

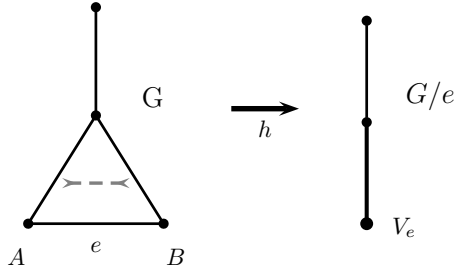
\begin{figure}[h]
  \centering
  \scalebox{0.9}
  {\begin{pspicture}(-1.2,-1)(5,3.7)
  \psline[linewidth=0.04cm,linecolor=black]{*-*}(-1,0)(0,1.6)(0,3.2)
  \psline[linewidth=0.04cm,linecolor=black]{*-*}(1,0)(0,1.6)
  \psline[linewidth=0.04cm,linecolor=black]{-}(-1,0)(1,0)
  \psline[linewidth=0.05cm,linecolor=gray,linestyle=dashed]{>-<}(-0.5,0.6)(0.5,0.6)
  \uput[-90](0,-0.1){$e$}
  \uput[-120](-1,-0.2){$A$}
  \uput[-60](1,-0.2){$B$}
  \uput[20](0.5,1.75){\Large{G}}
  \psline[linewidth=0.04cm,fillcolor=black,dotsize=0.07055555cm 2.0]{*-*}(4,1.5)(4,3)
  \psline[linewidth=0.06cm,fillcolor=black,dotsize=0.07055555cm 2.0]{-*}(4,1.5)(4,0)
  \uput[-10](4.2,0){$V_e$}
  \uput[20](4.4,1.75){\Large{$G/e$}}
  \psline[linewidth=0.08cm,fillcolor=black]{->}(2,1.7)(3,1.7)
  \uput[-90](2.5,1.7){$h$}
  \end{pspicture}}
  \caption{The map $h$.}
  \label{fig:OutlineH}
\end{figure}

Denote by $J(G)$ the set of vertices and midpoints of edges in $G$.

As usual, by \emph{cycle} we mean a simple closed curve, \emph{i.e.}, a path with different vertices,
unless the last one, which is equal to the first vertex.

Given $e\in E(G)$, denote by $\mathcal{C}(G,e)$ the set of cycles in $G$ with length $3$ containing $e$.

\begin{proposition}\label{l:quasi_isom}
  Let $G$ be a graph and $e\in E(G)$. Then
   \begin{equation}\label{eq:+3/2}
     d_{G/e}\big(h(x),h(y)\big) \le d_{G}(x,y) \le d_{G/e}\big(h(x),h(y)\big) + \frac32\,, \quad \forall x,y\in G.
   \end{equation}
  Furthermore, if $y\in J(G)$ or $x,y$ are not contained in the same cycle $C \in \mathcal{C}(G,e)$, then
   \begin{equation}\label{eq:+1}
     d_{G/e}\big(h(x),h(y)\big) \le d_{G}(x,y) \le d_{G/e}\big(h(x),h(y)\big) + 1.
   \end{equation}
\end{proposition}

\begin{proof}
Fix $x,y \in G$. Let $[xy]_G$ be a geodesic in $G$ joining $x$ and $y$.
Clearly, $h([xy]_G)$ is a path joining $h(x)$ and $h(y)$ with length at most $L([xy]_G)$, thus, we obtain $d_{G/e}\big(h(x),h(y)\big) \le d_{G}(x,y)$.
Hence, the first inequalities in \eqref{eq:+3/2} and \eqref{eq:+1} hold.

Let $\g'$ be a geodesic in $G/e$ joining $h(x)$ and $h(y)$.
If $x,y$ are not contained in the same cycle $C \in \mathcal{C}(G,e)$,
then there is a path $\g$ in $G$ with $h(\g)=\g'$ and $L(\g)\le L(\g')+1$ since $e$ (or a subset of $e$) can be included in $\g$.
Therefore, $d_{G}(x,y) \le L(\g)\le L(\g')+1 = d_{G/e}\big(h(x),h(y)\big) + 1$.

If $x,y \in C \in \mathcal{C}(G,e)$, then
$$
d_{G}(x,y)
\le d_{C}(x,y)
\le \frac12\, L(C)
= \frac32
\le d_{G/e}\big(h(x),h(y)\big) + \frac32\,.
$$

Finally, consider $x,y \in C \in \mathcal{C}(G,e)$ with $y \in J(G)$.

We deal with several cases.

\smallskip

\emph{Case} $1$.
If $d_{G/e}\big(h(x),h(y)\big) \ge 1/2$, then
$$
d_{G}(x,y)
\le \frac12\, L(C)
= \frac12 +1
\le d_{G/e}\big(h(x),h(y)\big) + 1.
$$

\emph{Case} $2$.
Assume that $d_{G/e}\big(h(x),h(y)\big) < 1/2$.

\smallskip

\emph{Case} $2.1$. $y$ is the midpoint of $e$.
If $x\in e$, then $d_{G}(x,y) \le 1/2$.
If $x\notin e$, then $d_{G}(x,y) = d_{G/e}\big(h(x),h(y)\big) + 1/2$.

\smallskip

\emph{Case} $2.2$. $y\in \{A,B\}$.
If $x\in e$, then $d_{G}(x,y) \le 1$.
If $x\notin e$, then we have either
$d_{G}(x,y) = d_{G/e}\big(h(x),h(y)\big)$ or $d_{G}(x,y) = d_{G/e}\big(h(x),h(y)\big) + 1$.

\smallskip

\emph{Case} $2.3$. $y\in V(C) \setminus \{A,B\}$.
Thus $d_{G}(x,y) = d_{G/e}\big(h(x),h(y)\big)$.

\smallskip

\emph{Case} $2.4$. $y\in J(G) \setminus \{V(C) \cup e\}$.
We have either
$d_{G}(x,y) = d_{G/e}\big(h(x),h(y)\big)$, $d_{G}(x,y) = d_{G/e}\big(h(x),h(y)\big) + 1$ or $d_{G}(x,y) = 1 - d_{G/e}\big(h(x),h(y)\big)\le 1$.

This finishes the proof.
\end{proof}

Note that the inequalities in \eqref{eq:+3/2} and \eqref{eq:+1} are attained.
If $G$ is any graph, $[v,w]\in E(G)$ and $\{v,w\}\cap \{A,B\} = \emptyset$, then $d_{G/e}(h(v),h(w))=1=d_G(v,w)$.
Consider a cycle graph $G=C_3$ and $x,y\in C_3$ such that $x\neq y$ and there is $v\in V(C_3)$ with $d_{C_3}(x,v)=d_{C_3}(v,y)=3/4$.
Let $e$ be the edge in $C_3$ with $x,y\notin e$.
Hence, we have $d_{C_3}(x,y)=3/2$ and $h(x)=h(y)$.
Finally, consider a cycle graph $G=C_3$, $x_0,y_0$ two different midpoints of edges in $C_3$ and $e \in E(C_3)$ with $x_0,y_0\notin e$.
Thus, $d_{C_3}(x_0,y_0)=1$ and $h(x_0)=h(y_0)$.

\medskip

The previous lemma has the following consequence about the continuity of $h$.

\begin{proposition}\label{prop:Lips}
 The map $h$ is an $1$-Lipschitz continuous function.
\end{proposition}

Let $(X,d_X)$ and $(Y,d_Y)$  be two metric spaces. A map $f: X\longrightarrow Y$ is said to be
an $(\alpha, \beta)$-\emph{quasi-isometric embedding}, with constants $\alpha\geq 1,\
\beta\geq 0$ if for every $x, y\in X$:
$$
\alpha^{-1}d_X(x,y)-\beta\leq d_Y(f(x),f(y))\leq \alpha d_X(x,y)+\beta.
$$
The function $f$ is $\varepsilon$-\emph{full} if
for each $y \in Y$ there exists $x\in X$ with $d_Y(f(x),y)\leq \varepsilon$.

\smallskip

A map $f: X\longrightarrow Y$ is said to be
a \emph{quasi-isometry} if there exist constants $\alpha\geq 1,\
\beta,\varepsilon \geq 0$ such that $f$ is an $\varepsilon$-full
$(\alpha, \beta)$-quasi-isometric embedding.

\smallskip

A fundamental property of hyperbolic spaces is the following:

\begin{theorem}[Invariance of hyperbolicity]\label{invarianza}
Let $f:X\longrightarrow Y$ be an $(\alpha,\beta)$-quasi-isometric embedding between the geodesic
metric spaces $X$ and $Y$.
If $Y$ is hyperbolic, then $X$ is hyperbolic.
Furthermore, if $Y$ is $\d$-hyperbolic, then $X$ is $\d'$-hyperbolic, where $\d'$ is a constant which just depends on
$\alpha,\beta,\d$.

Besides, if $f$ is $\e$-full for some $\e\ge0$ (a quasi-isometry), then $X$ is
hyperbolic if and only if $Y$ is hyperbolic.
Furthermore, if $X$ is $\d$-hyperbolic, then $Y$ is $\d'$-hyperbolic, where $\d'$ is a constant which just depends on
$\alpha,\beta,\d,\e$.
\end{theorem}

Using the invariance of hyperbolicity (Theorem \ref{invarianza}), we can obtain the main qualitative result in this section.

\begin{theorem}
\label{t:quali}
Let $G$ be a graph and $e\in E(G)$. Then
$G$ is hyperbolic if and only if $G/e$ is hyperbolic.
Furthermore, if $G$ (respectively, $G/e$) is
$\d$-hyperbolic, then $G/e$ (respectively, $G$) is $\d'$-hyperbolic, where $\d'$ is a
constant which just depends on $\d$.
\end{theorem}

\begin{proof}
Proposition \ref{l:quasi_isom} gives that $h$ is a $0$-full $(1,3/2)$-quasi-isometry from $G$ onto
$G/e$, and  we obtain the result by Theorem \ref{invarianza}.
\end{proof}

One can expect that the edge contraction is a monotone transformation for the hyperbolicity constant,
\emph{i.e.}, the hyperbolicity constant always decreases by edge contraction
(for instance, if $e$ is any edge of the cycle graph $C_3$, then $C_3/e$ is the path graph $P_2$ and $\d(P_2)=0<3/4=\d(C_3)$;
if $e$ is any edge of the cycle graph $C_n$ with $n\ge 4$, then $C_n/e$ is the cycle graph $C_{n-1}$ and $\d(C_{n-1})=(n-1)/4<n/4=\d(C_n)$).
However, the following example provides a family of graphs such that the hyperbolicity constant increases by contracting certain edge.
Recall that the \emph{girth} of a graph $G$ is the minimum of the lengths of the cycles in $G$.

\begin{example}\label{ex:wn-10}
In \cite[Theorem 11]{RSVV}, the authors obtain the precise value of the hyperbolicity constant of the wheel graph
with $n$ vertices $W_n$: $\d(W_4)=\d(W_5)=1$, $\d(W_n)=3/2$ for every $7\le n\le 10$,
and $\d(W_n)=5/4$ for $n=6$ and for every $n\ge 11$.
Note that we can obtain $W_n$ from $W_{n+1}$ by edge contraction, so, we have that $\d(W_{11})=5/4$ and $\d(W_{10})=3/2$.
Furthermore, in \cite{CRS2} the authors obtain the value of the hyperbolicity constant of the graph join of $E$
(the empty graph with just one vertex) and every graph.
Thus, taking $G$ as a graph join of $E$ and any graph $G^*$ with girth $10$, then $\d(G)=5/4$,
but contracting an edge $e$ belonging to any cycle in $G^{*}$ with length 10,
$G/e$ is the graph join of $E$ and other graph with girth $9$, so \cite[Corollary 3.23]{CRS2} gives $\d(G/e)=3/2$.
\end{example}

Other aim in this work is to obtain quantitative relations between $\d(G/e)$ and $\d(G)$.
Since the proofs of these inequalities are long, in order to make the arguments more transparent, we collect some results in technical lemmas.

\smallskip

For any simple (non-selfintersecting)
path $\g'$ joining two different points in $G/e$ which are not contained in an edge $e_0$ with $h^{-1}(e_0) \in \mathcal{C}(G,e)$,
we define $\G(\g')$ as the set of paths $\g$ in $G$ such that $h(\g)=\g'$ and
$$
\G_0(\g') = \big\{ g\in \G(\g')\,|\,\, L(g)\le L(\g) \;\;\, \forall \,\g\in \G(\g') \big\}.
$$
Note that any curve in $\G_0(\g')$ is a simple path.
We denote by $h_0^{-1}(\g')$ any fixed choice of curve in $\G_0(\g')$.
If $t'\in \g'\setminus {V_e}$ we denote by $t=h_0^{-1}(t')$ the point in $h_0^{-1}(\g')$ such that $h(t)=t'$
(note that, since $\g'$ is a simple path, any $t'\in \g'\setminus {V_e}$ defines an unique $t\in h_0^{-1}(\g')$).
If $t'=V_e \in \g'$, then $h_0^{-1}(V_e)=h^{-1}(V_e)=e$.
Hence, $h_0^{-1}(t')=h^{-1}(t') \cap h_0^{-1}(\g')$.
Furthermore, for any geodesic $\g'$ in $G/e$ such that $V_e\notin \g'$ we have that
$$
h\!\mid_{h^{-1}(\g')} : h^{-1}(\g') \longrightarrow \g'
$$
is a bijective map and $\g = h^{-1}(\g') = h_0^{-1}(\g')$ is a geodesic in $G$ with $L(\g)=L(\g')$.

\begin{lemma}\label{l:dist2_G}
   Let $G$ be a graph and $e \in E(G)$.
   Let $x,y\in G\setminus\{e\}$ such that there is no $C \in \mathcal{C}(G,e)$ with $x,y \in C$.
   Assume that there are two geodesics $\g_G$ and $\g_{G/e}$ in $G$ and $G/e$,
   respectively, joining $x,y$ and $h(x),h(y)$, respectively, such that $L(\g_G)=L(\g_{G/e})=L(h(\g_G))$ and $e\subset h_0^{-1}( \g_{G/e} )$. Then we have
   \begin{equation}\label{eq_dist2_xyG}
     d_{G/e}(h(\a),\g_{G/e} )\le \d(G) \quad \forall\ \a\in \g_{G}
   \end{equation}
   and
   \begin{equation}\label{eq_dist2_xyG/e}
     d_{G/e}\big( \a',h(\g_{G}) \big)\le 2 \d(G) \quad \forall\ \a'\in \g_{G/e}.
   \end{equation}
\end{lemma}

\begin{remark}
Since there is no $C \in \mathcal{C}(G,e)$ with $x,y \in C$, we deduce that
$h(x),h(y)$ are not contained in an edge $e_0$ with $h^{-1}(e_0) \in \mathcal{C}(G,e)$,
and then $h_0^{-1}(\g_{G/e})$ is well defined.
\end{remark}

\begin{proof}
Without loss of generality we can assume that $G$ is hyperbolic, since otherwise the inequalities trivially hold.
Let $z$ be the midpoint of $e=[A,B]$. By symmetry, we can assume that the
closure of the connected components of
$h_0^{-1}(\g_{G/e}\setminus \{ V_e \})$ join $x$ with $A$, and $B$ with $y$.
Clearly,
$$
d_{G}(z,x) = d_{G}(z,A)+\frac12 =d_{G/e}(h(x),V_e)+\frac12\,,
\qquad
d_{G}(z,y) = d_{G}(y,B)+\frac12 =d_{G/e}(h(y),V_e)+\frac12\,.
$$
So, there are geodesics $[xz]_G$ and $[zy]_G$ in $G$ verifying the following: $[xz]_G$ contains the
closure of the connected component of
$h_0^{-1}(\g_{G/e}\setminus \{ V_e \})$ joining $x$ with $A$, and $[zy]_G$ contains
the closure of the connected component of $h_0^{-1}(\g_{G/e}\setminus \{ V_e \})$ joining $B$ with $y$.
Hence, $T:=\{\g_G,[yz]_G,[zx]_G\}$ is a geodesic triangle in $G$ and so,
$$
d_G(\a,[yz]_G\cup [zx]_G) = d_G(\a,h_0^{-1}(\g_{G/e})) \leq \d(G)
$$
for every $\a \in \g_G$,
and \eqref{eq_dist2_xyG} holds by Proposition \ref{l:quasi_isom}.

In order to obtain \eqref{eq_dist2_xyG/e}, without loss of generality we can assume that
$\a'\in h([yz]_{G})$, since $\g_{G/e}=h([yz]_{G}) \cup h([zx]_{G})$.
Denote by $\a$ a point in $[yB]_G$ with $h(\a)=\a'$.
If $d_{G}(y, e) = d_G(y,B)\leq 2\d(G)$, then we have
$$
d_{G/e}(\a', h(\g_{G}))\le d_{G}(\a, \g_{G})\le d_G(\a,y)\le d_G(B,y)\leq 2\d(G).
$$
Assume that $d_{G}(y, e)> 2\d(G)$.
Now we can take a point $w\in[yz]_{G}$ such that $d_{G}(w , e)=\d(G)$.
If $\a\in [wy]\setminus\{w\}$, then the hyperbolicity of $G$ implies $d_{G}(\a, \g_G \cup [zx]_{G}) \leq \d(G)$;
note that $d_{G}(\a , [zx]_{G}) > \d(G)$ since $\g_{G/e}$ is a geodesic in $G/e$ and $d_{G}(w , e)=\d(G)$.
Hence, $d_G(\a,\g_G)\leq \d(G)$ and thus
Proposition \ref{l:quasi_isom} gives
$d_{G/e}(\a',h(\g_G)) \leq \d(G)$.
Assume now that $\a \in [wB]\setminus{B}$ (therefore, $\a'\neq V_e$).
Thus, there exists $\a_1\in[yz]_{G}$ such that $d_{G}(\a,\a_1) = \d(G)$ and $d_{G}(\a_1 , e) > \d(G)$.
Therefore, $\a_1\in [wy]\setminus\{w\}$ and we have proved that $d_G(\a_1,\g_G)\leq \d(G)$.
Hence, Proposition \ref{l:quasi_isom} gives
$$
d_{G/e}(\a' , h(\g_G))
\le d_{G}(\a , \g_G)
\le d_{G}(\a , \a_1) + d_{G}(\a_1 , \g_G)
\le 2\d(G).
$$
If $\a=B$, then $\a' = V_e$
and the inequality for $\a'=V_e$ is obtained by continuity.
\end{proof}

In what follows,
if $x,y$ belong to some $C \in \mathcal{C}(G,e)$, then we denote by $h_0^{-1}( [h(x)h(y)]_{G/e} )$ any fixed choice of a geodesic in $C$ with
$[h(x)h(y)]_{G/e} \subseteq h(h_0^{-1}( [h(x)h(y)]_{G/e} ))$, and by $h_0^{-1}( \a' )$ any point in $h^{-1}( \a' ) \cap h_0^{-1}( [h(x)h(y)]_{G/e} )$.

\begin{lemma}\label{l:dist1_G/e}
   Let $G$ be a graph and $e\in E(G)$ such that $G/e$ is not a tree.
   Let $[xy]_G$ be a geodesic in $G$ joining $x,y \in J(G)$.
   Assume that $h([xy]_G)$ is not a geodesic in $G/e$ and let $[h(x)h(y)]_{G/e}$ be a geodesic in $G/e$ joining $h(x)$ and $h(y)$.
      Then we have
   \begin{equation}\label{eq_dist1_xyG}
     d_{G}(h_0^{-1}(\a'),[xy]_G)\le \d(G/e)+1 \le \frac73 \,\d(G/e), \quad \forall\ \a'\in [h(x)h(y)]_{G/e}
   \end{equation}
   and
   \begin{equation}\label{eq_dist1_xyG/e}
     d_{G}\big(\a,h_0^{-1}([h(x)h(y)]_{G/e}) \big)\le 2 \d(G/e), \quad \forall\ \a\in [xy]_{G}.
   \end{equation}
\end{lemma}

\begin{proof}
Without loss of generality we can assume that $G/e$ is hyperbolic, since otherwise the inequalities hold.
Since $G/e$ is not a tree, $\d(G/e)\ge 3/4$ by \cite[Theorem 11]{MRSV}, and
$$
\d(G/e)+1 \le \frac73 \,\d(G/e).
$$

We deal with several cases.

\smallskip

$(a)$ There exists $C \in \mathcal{C}(G,e)$ with $x,y \in C$.
For every $\a,\b\in C$ we have
$$
d_G(\a,\b)
\le \diam C
= \frac32
= \min \Big\{ \frac34+1, 2 \frac34 \, \Big\}
\le \min \big\{ \d(G/e)+1, 2 \d(G/e) \big\}.
$$
Since $[xy]_G$ and $h_0^{-1}([h(x)h(y)]_{G/e}$ are contained in $C$,
\eqref{eq_dist1_xyG} and \eqref{eq_dist1_xyG/e} hold.

\smallskip

$(b)$ There is no $C \in \mathcal{C}(G,e)$ with $x,y \in C$.
Thus $h(x),h(y)$ are not contained in an edge $e_0$ with $h^{-1}(e_0) \in \mathcal{C}(G,e)$,
and then $h_0^{-1}([h(x)h(y)]_{G/e})$ is defined as before Lemma \ref{l:dist2_G}.

Let $[A,B]:=e$.
Since $h([xy]_G)$ is not a geodesic in $G/e$ and $x,y \in J(G)$, we have
$e \cap [xy]_G \subsetneq \{A,B\}$, $L([xy]_G)\ge 3/2$, $d_G(x,y)=d_{G/e}(h(x),h(y))+1$
and $V_e \in [h(x)h(y)]_{G/e}$.

\smallskip

$(b.1)$
Assume that there exists a cycle $C\in \mathcal{C}(G,e)$ with $L([xy]_G\cap C)> 1$.
Since $[xy]_G$ is a geodesic in $G$, $h([xy]_G)$ is not a geodesic in $G/e$ and $x,y \in J(G)$, we have $L([xy]_G\cap C) = 3/2$,
$C \subset h_0^{-1}([h(x)h(y)]_{G/e}) \cup [xy]_G$,
the closures of $h_0^{-1}([h(x)h(y)]_{G/e}) \setminus C$ and $[xy]_G \setminus C$ are two geodesics in $G$ with the same endpoints,
and the closures of $[h(x)h(y)]_{G/e} \setminus h(C)$ and $h([xy]_G) \setminus h(C)$ are two geodesics in $G/e$ with the same endpoints.
Since $L([xy]_G\cap C) = 3/2= L(h_0^{-1}([h(x)h(y)]_{G/e}) \cap C)$ and $L(C)=3$, we have
$$
\begin{aligned}
d_{G}(h_0^{-1}(\a'),[xy]_G) \le 3/4 \le \d(G/e), \quad & \forall\ \a'\in [h(x)h(y)]_{G/e} \cap h(C),
\\
d_{G}\big(\a,h_0^{-1}([h(x)h(y)]_{G/e}) \big) \le 3/4 \le \d(G/e), \quad & \forall\ \a\in [xy]_{G} \cap C.
\end{aligned}
$$
Since the closures of $h_0^{-1}([h(x)h(y)]_{G/e}) \setminus C$ and $[xy]_G \setminus C$ are two geodesics in $G$ with the same endpoints,
and the closures of $[h(x)h(y)]_{G/e} \setminus h(C)$ and $h([xy]_G) \setminus h(C)$ are two geodesics in $G/e$ with the same endpoints,
we also have
$$
d_{G}(h_0^{-1}(\a'),[xy]_G)
\le d_{G}\big(h_0^{-1}(\a'),[xy]_G \setminus C\big)
= d_{G/e}\big(\a',h([xy]_G) \setminus h(C)\big)
\le \d(G/e),
$$
for every $\a'\in [h(x)h(y)]_{G/e} \setminus h(C)$, and
$$
d_{G}\big(\a,h_0^{-1}([h(x)h(y)]_{G/e}) \big)
\le d_{G}\big(\a,h_0^{-1}([h(x)h(y)]_{G/e}) \setminus C \big)
= d_{G/e}\big(h(\a), [h(x)h(y)]_{G/e} \setminus h(C) \big)
\le \d(G/e) ,
$$
for every $\a\in [xy]_{G} \setminus C$.

\smallskip

$(b.2)$
Assume now that $L([xy]_G\cap C)\le 1$ for every $C\in \mathcal{C}(G,e)$.

Note that $L\big( h([xy]_G) \big) = L([xy]_G)$,
since $e \cap [xy]_G \subset \{A,B\}$,
so, for any $z\in[xy]_G$ we have $L\big( h([xz]_G)\big)=L([xz]_G)$ and $L\big( h([zy]_G)\big)=L([zy]_G)$.
Consider the points $A',B'\in [xy]_G$ such that $d_{G}(x,A')=d_G(x,e)$ and $d_{G}(y,B')=d_G(y,e)$. Since $x,y\in J(G)$, we have  $A',B'\in V(G)$.
Since $L\big( h([xy]_G) \big) = L([xy]_G)$, $[A',B']\in E(G)$ and $[A',B']\subset[xy]_G$.
Let $z$ be the midpoint of $[A',B']$.
Since $d(x,A')=d(x,e)$, $d_G(y,B')=d_G(y,e)$, $d_G(z,A')=d_G(z,B')=1/2$ and $L([xy]_G\cap C)\le 1$ for every $C\in \mathcal{C}(G,e)$,
we have that $h([xz]_G)$ and $h([zy]_G)$ are geodesics in $G/e$.
Hence, $T=\{[h(x)h(y)]_{G/e},h([yz]_{G}),h([zx]_{G})\}$ is a geodesic triangle in $G/e$, and thus
$$
\begin{aligned}
d_{G}(h_0^{-1}(\a'),[xy]_G) &
\le
d_{G/e}(\a',h([xy]_G)) + 1
= d_{G/e}(\a',h([xz]_G)\cup h([zy]_G)) + 1 \\
&
\le \d(T) + 1 \le \d(G/e) + 1 \le \frac73 \, \d(G/e),
\end{aligned}
$$
for every $\a'\in [h(x)h(y)]_{G/e}$.

\smallskip

In order to obtain \eqref{eq_dist1_xyG/e}, without loss of generality we can assume that $\a\in [yz]_{G}$.
If $L([yz]_{G})\leq 2\d(G/e)$, then we have $d_G\big(\a, h_0^{-1}([h(x)h(y)]_{G/e})\big)\le d_G(\a,y)\le 2\d(G/e)$. Assume that $L([yz]_{G}) > 2\d(G/e)$.
Let $w$ be the point in $[yz]_{G}$ with $d_{G}(w , z)=\d(G/e)$.
If $\a\in [wy]_G\setminus\{w\}$, then the hyperbolicity of $G/e$ implies $d_{G/e}\big( h(\a), [h(x)h(y)]_{G/e} \cup h([zx]_{G}) \big) \leq \d(G/e)$.
Note that if $d_{G/e}\big( h(\a) , h([zx]_{G}) \big) \leq \d(G/e)$, then a geodesic $\g$ joining $h(\a)$ and $h([zx]_{G})$ in $G/e$ contains $V_e$ and,
since $V_e\in [h(x)h(y)]_{G/e}$, we obtain
$$
d_{G/e}(h(\a) , [h(x)h(y)]_{G/e}) \leq d_{G/e}(h(\a) , V_e) \leq L(\g) \leq \d(G/e).
$$
Thus, we have $d_{G/e}(h(\a) , [h(x)h(y)]_{G/e}) \leq \d(G/e)$.
Hence, we obtain
$$
d_G\big( \a, h_0^{-1}([h(x)h(y)]_{G/e}) \big)=d_{G/e}(h(\a) , [h(x)h(y)]_{G/e}) \leq \d(G/e).
$$
Assume now that $\a \in [zw]_G\setminus\{z\}$.
Then, there exists $\a_1\in[wy]_{G}\setminus\{w\}$ such that $d_{G}(\a,\a_1) = \d(G/e)$, and we deduce
$$
d_G\big( \a, h_0^{-1}([h(x)h(y)]_{G/e}) \big)
\le d_{G}(\a , \a_1) + d_G\big( \a_1, h_0^{-1}([h(x)h(y)]_{G/e}) \big)
\leq 2\d(G/e).
$$
The inequality for $\a=z$ is obtained by continuity.
\end{proof}

\begin{remark}\label{r:No3e_G}
   Let $G$ be a graph, $e\in E(G)$ and $T=\{\g_1,\g_2,\g_3\}$ a geodesic triangle in $G$. Then at least one of the curves $h(\g_1),h(\g_2),h(\g_3)$ is a geodesic in $G/e$, since otherwise there exists another geodesic triangle $T'=\{\g'_1,\g'_2,\g'_3\}$ with the same vertices that $T$ and such that the edge $e$ is contained in $\g_1' \cap \g_2' \cap \g_3'$.
\end{remark}

The following result will be useful.

\begin{theorem}\cite[Theorem 2.7]{BRS}
\label{t:TrianVMp}
For a hyperbolic graph $G$, there exists
a geodesic triangle $T = \{x, y, z\}$ that is a cycle with $x, y, z \in J(G)$ and $\d(T) = \d(G)$.
\end{theorem}

In order to prove Theorem \ref{t:quanti} we will need the following technical result.

\begin{lemma} \label{l:x}
Let $G$ be a graph and $e\in E(G)$. If $G/e$ is a tree, then
$$
\d(G) \le 1.
$$
\end{lemma}

\begin{proof}
If $G$ is a tree, then $\d(G) =0\le 1$.

Assume now that $G$ is not a tree.
Since $G/e$ is a tree, if an edge $e_0$ is contained in a cycle in $G$, then it is contained in some cycle $C_0 \in \mathcal{C}(G,e)$
and it contains $A$ or $B$.
Since any cycle in $\mathcal{C}(G,e)$ contains the edge $e$, if $[A,B]=e$, then every cycle in $G$ contains the vertices $A$ and $B$.
Therefore, any cycle in $G$ has length at most $4$.

Theorem \ref{t:quali} gives that $G$ is hyperbolic since $\d(G/e) =0$.
Hence, by Theorem \ref{t:TrianVMp} there exist a geodesic triangle $T=\{x,y,z\}$ in $G$ that is a cycle with $x,y,z\in J(G)$ and $p\in [xy]$ with $d_{G}(p,[yz]\cup[zx]) = \d(T)=\d(G)$.
Since $T$ is a cycle, we have seen that $L(T)\le 4$.
Thus
$$
\d(G) = d_{G}(p,[yz]\cup[zx])
\le d_{G}(p,\{x,y\})
\le \frac12\, d_{G}(x,y)
\le \frac14\, L(T)
\le 1.
$$
\end{proof}

The previous results allow to obtain a quantitative version of Theorem \ref{t:quali}.

\begin{theorem}
\label{t:quanti}
Let $G$ be a graph and $e\in E(G)$. Then
\begin{equation}\label{eq:quanti}
  \frac13\,\d(G/e) \le \d(G) \le \frac{16}3 \,\d(G/e) + 1.
\end{equation}
\end{theorem}

\begin{proof}
By Theorem \ref{t:quali} we have that $G$ and $G/e$ are hyperbolic or not simultaneously.
If $G$ and $G/e$ are not hyperbolic, then $\d(G)=\d(G/e)=\infty$ and \eqref{eq:quanti} holds.
Assume now that both graphs are hyperbolic.

\smallskip

Let us prove the first inequality in \eqref{eq:quanti}.
If $G$ is a tree, then $\d(G/e)=0$ and the first inequality holds.
Assume that $G/e$ is not a tree, thus $\d(G/e)>0$.
By Theorem \ref{t:TrianVMp} there exist a geodesic triangle $T'=\{[x'y'],[y'z'],[z'x']\}$ in $G/e$ that is a cycle with
$x',y',z'\in J(G/e)$ and $p'\in [x'y']$ with $d_{G/e}(p',[y'z']\cup[z'x']) = \d(T')=\d(G/e)$.

Consider $T \subset h^{-1}([x'y'] \cup [y'z'] \cup [z'x'])$ such that $T$ is a cycle with $h(T)=T'$.
Define  $x:=h^{-1}(x')\cap T$, $y:=h^{-1}(y')\cap T$ and $z:=h^{-1}(z')\cap T$, if $V_e\notin\{x',y',z'\}$;
otherwise, if $V_e=a'$ with $a'\in\{x',y',z'\}$, then we define $a$ as the midpoint of $e$.
Hence, we can define $g_{ab}$ as the simple curve contained in $T$ joining $a$ and $b$,
and such that $h(g_{ab})=[a'b']$, for $a,b \in \{x,y,z\}$
(note that $g_{ab}=h_0^{-1}([a'b'])$ if $h_0^{-1}([a'b'])$ is defined as before Lemma \ref{l:dist2_G}, \emph{i.e.}, if $a',b'$
are not contained in an edge $e_0$ with $h_0^{-1}(e_0) \in \mathcal{C}(G,e)$).
Then $x,y,z\in J(G)$ and $T$ can be seen as the triangle $\{ g_{xy},g_{yz},g_{zx} \}$.
Note that if $V_e \in \{ x',y',z' \}$, then $T$ is a geodesic triangle in $G$.

We deal with several cases.

\smallskip

$(a)$
If $T$ is a geodesic triangle in $G$, then by Proposition \ref{l:quasi_isom} we have for any $p \in h^{-1}(p') \cap g_{xy}$
$$
\d(G/e)= \d(T')=d_{G/e}(p',[y'z']\cup[z'x'])\le d_{G}( p,g_{yz}\cup g_{zx} ) \le \d(T) \le \d(G)
$$
and so, the first inequality in \eqref{eq:quanti} holds.

\smallskip

$(b)$
Assume that $T$ is not a geodesic triangle in $G$. Thus, we have $V_e \in T'\setminus \{x',y',z'\}$, $e\subset T$ and $L(T)=L(T')+1$.
Since $V_e \notin \{x',y',z'\}$, the edge $e$ is contained in exactly one of
$g_{xy}, g_{yz}, g_{zx},$ and the other two paths  are geodesics in $G$ by Proposition \ref{l:quasi_isom}.

\smallskip

$(b.1)$
Assume that $e\subset g_{xy}$.
Note that $e$ is contained in the interior of $g_{xy}$ (recall that $V_e \notin \{x',y',z'\}$);
since $x,y \in J(G),$ we have $L(g_{xy}) \ge 2$.
Therefore, $x,y\in G\setminus\{e\}$ and there is no $C \in \mathcal{C}(G,e)$ with $x,y \in C$.
Hence,
$h_0^{-1}([x'y'])$ is defined as before Lemma \ref{l:dist2_G},
$g_{xy} = h_0^{-1}( [x'y'] )$ and $e\subset h_0^{-1}( [x'y'] )$.
Consider  a geodesic $[xy]$ in $G$ joining $x$ and $y$.
We have $L([xy])=L([x'y'])=L(h([xy]))$.
Note that $g_{xy}$ is not a geodesic by hypothesis.
By Lemma \ref{l:dist2_G} there is $p\in[xy]$ such that $d_{G/e}(p',h(p))\le 2\d(G)$.
Thus, since $\{[xy],g_{yz},g_{zx}\}$ is a geodesic triangle in $G$,
there is $p_1\in g_{yz} \cup g_{zx}$ such that $d_G(p,p_1)\le \d(G)$.
Hence, Proposition \ref{l:quasi_isom} gives
$$
\begin{aligned}
\d(G/e) &
= d_{G/e}(p',[y'z']\cup[z'x']) \leq \d_{G/e}(p',h(p_1))
\le d_{G/e}(p',h(p)) + d_{G/e}(h(p),h(p_1))
\\
& \le 2\d(G) + d_G(p,p_1) \le 3\d(G).
\end{aligned}
$$

$(b.2)$
Assume now that $e\subset g_{yz} \cup g_{zx}$.
By symmetry, we can assume that $e\subset g_{yz}$.
Note that $g_{yz}$ is not a geodesic by hypothesis, and that $g_{xy}, g_{zx}$ are geodesics.
Consider a geodesic $[yz]$  in $G$ joining $y$ and $z$.
Since $\{g_{xy},[yz], g_{zx}\}$ is a geodesic triangle in $G$,
there is $p\in [yz]\cup g_{zx}$ such that $d_G(h_0^{-1}(p'),p)\le \d(G)$.

\smallskip

$(b.2.1)$
If $p\in g_{zx}$, then Proposition \ref{l:quasi_isom} gives
$$
\d(G/e) = d_{G/e}(p',[y'z']\cup[z'x']) \le d_{G/e}(p', h(p)) \le d_G(h_0^{-1}(p'),p)\le \d(G).
$$

$(b.2.2)$
Assume that $p\in [yz]$.
The argument in $(b.1)$ also gives that
$y,z\in G\setminus\{e\}$, there is no $C \in \mathcal{C}(G,e)$ with $y,z \in C$,
$e\subset g_{yz}=h_0^{-1}( [y'z'] )$, and $L([yz])=L([y'z'])=L(h([yz]))$.
Therefore, by Lemma \ref{l:dist2_G} there is $p'_1\in [y'z']$ such that $d_{G/e}(h(p),p'_1)\le \d(G)$.
Thus, we have by Proposition \ref{l:quasi_isom}
$$
\begin{aligned}
\d(G/e)
& =d_{G/e}(p',[y'z']\cup[z'x'])\le d_{G/e}(p',p_1') \le d_{G/e}(p', h(p)) + d_{G/e}(h(p),p'_1)
\\
& \le d_G(h_0^{-1}(p'),p) + \d(G) \le 2\d(G).
\end{aligned}
$$
Hence, the first inequality in \eqref{eq:quanti} holds.

\medskip

Let us prove the second inequality in \eqref{eq:quanti}.
By Theorem \ref{t:TrianVMp} there exist a geodesic triangle $T=\{ x,y,z \}$ in $G$ that is a cycle
with $x,y,z\in J(G)$ and $p\in[xy]$ with $d_{G}(p,[yz]\cup[zx])=\d(T)=\d(G)$.
Since $x,y,z\in J(G)$ we have
$$
d_{G}(p,[yz]\cup[zx])= d_{G}(p,J(G)\cap ([yz]\cup[zx])),
$$
and if $d_{G}(p,[yz]\cup[zx])= d_{G}(p,q)$ with $q\in J(G)\cap ([yz]\cup[zx])$, then Proposition \ref{l:quasi_isom} gives
$$
d_{G}(p,[yz]\cup[zx])= d_{G}(p,q) \le d_{G/e}\big( h(p) , h(q) \big) + 1.
$$

If $\d(G)\le 1$, then the second inequality in \eqref{eq:quanti} holds.
Hence, we can assume that $\d(G)>1$.
Note that since $\d(G)>1$ we have that $G/e$ is not a tree by Lemma \ref{l:x}.

Let $n$ be the number of geodesics in $G/e$ of the set $\big\{h([xy]),h([yz]),h([zx])\big\}$.
By Remark \ref{r:No3e_G} we have $n\in\{1, 2, 3\}$.

We consider several cases.

\smallskip

$(A)$
Assume that $h([xy])$ is a geodesic in $G/e$.

\smallskip

$(A.1)$
If $n=3$, then Proposition \ref{l:quasi_isom} gives
$$
\d(G)=d_{G}(p,[yz]\cup[zx])\le d_{G/e}\big( h(p) , h([yz])\cup h([zx]) \big) + 1  \le \d(G/e)+1.
$$

$(A.2)$
Consider the case $n=2$. By symmetry we can assume that $h([yz])$ is a geodesic in $G/e$
and let $[h(z)h(x)]$ be a geodesic in $G/e$ joining $h(z)$ and $h(x)$.
Then there is $p'\in h([yz])\cup[h(z)h(x)]$ with $d_{G/e}(h(p), p')\le \d(G/e)$.
By Proposition \ref{l:quasi_isom}, we have $d_{G}(p,h_0^{-1}(p'))\le d_{G/e}\big( h(p) , p' \big) + 1\le \d(G/e) + 1$.
If $p'\in h([yz])$, then $h_0^{-1}(p')\in [yz]$ and
$$
\d(G)=d_{G}(p,[yz]\cup[zx])\le d_{G}(p,h_0^{-1}(p')) \le \d(G/e) + 1.
$$
Assume that $p'\in[h(z)h(x)]$.
By Lemma \ref{l:dist1_G/e}, there is $p_1\in[zx]$ with $d_G(h_0^{-1}(p'),p_1)\le (7/3)\, \d(G/e)$.
Then we have
$$
\d(G)=d_{G}(p,[yz]\cup[zx])\le d_G(p,p_1) \le d_{G}(p,h_0^{-1}(p')) + d_G(h_0^{-1}(p'),p_1) \le \frac{10}3\,\d(G/e) + 1.
$$

$(A.3)$
If $n=1$, then let $[h(y)h(z)],[h(z)h(x)]$ be geodesics in $G/e$ joining $h(y),h(z)$ and $h(z),h(x),$ respectively.
Then there is $p'\in [h(y)h(z)]\cup[h(z)h(x)]$ with $d_{G/e}(h(p), p')\le \d(G/e)$.
By Proposition \ref{l:quasi_isom} we have $d_{G}(p,h_0^{-1}(p'))\le d_{G/e}\big( h(p) , p' \big) + 1\le \d(G/e) + 1$.
By symmetry we can assume that $p'\in[h(y)h(z)]$.
By Lemma \ref{l:dist1_G/e} there is $p_1\in[yz]$ with $d_G(h_0^{-1}(p'),p_1)\le (7/3)\, \d(G/e)$.
Thus,
$$
\d(G)=d_{G}(p,[yz]\cup[zx])\le d_G(p,p_1)\le d_{G}(p,h_0^{-1}(p')) + d_G(h_0^{-1}(p'),p_1) \le \frac{10}3 \,\d(G/e) + 1.
$$

$(B)$
Assume now that $h([xy])$ is not a geodesic in $G/e$. Let $[h(x)h(y)]$ be a geodesic in $G/e$ joining $h(x)$ and $h(y)$.
By Lemma \ref{l:dist1_G/e} there is $p'\in[h(x)h(y)]$ with $d_G(p,h_0^{-1}(p'))\le 2\d(G/e)$.

\smallskip

$(B.1)$
Consider the case $n=2$, \emph{i.e.}, $h([yz])$ and $h([zx])$ are geodesics in $G/e$.
Then there is $p''\in h([yz])\cup h([zx])$ with $d_{G/e}(p', p'')\le \d(G/e)$.
By Proposition \ref{l:quasi_isom}, we have
$$
d_{G}\big(h_0^{-1}(p'),h_0^{-1}(p'')\big)\le d_{G/e}( p' , p'' ) + 1\le \d(G/e) + 1.
$$
Thus, $h_0^{-1}(p'')\in [yz] \cup [zx]$ and
$$
\d(G)=d_{G}(p,[yz]\cup[zx])\le d_G(p,h_0^{-1}(p''))\le d_{G}(p,h_0^{-1}(p')) + d_G\big(h_0^{-1}(p'),h_0^{-1}(p'')\big) \le 3\d(G/e) + 1.
$$

$(B.2)$
Consider the case $n=1$.
By symmetry we can assume that $h([yz])$ is a geodesic in $G/e$ and let $[h(z)h(x)]$ be a geodesic in $G/e$ joining $h(z)$ and $h(x)$.
Then there is $p''\in h([yz])\cup[h(z)h(x)]$ with $d_{G/e}(p', p'')\le \d(G/e)$.
By Proposition \ref{l:quasi_isom}, we have
$$
d_{G}\big(h_0^{-1}(p'),h_0^{-1}(p'')\big)\le d_{G/e}( p' , p'') + 1 \le \d(G/e) + 1.
$$
If $p''\in h([yz])$, then we obtain
$$
\d(G)=d_{G}(p,[yz]\cup[zx])\le  d_G(p,h_0^{-1}(p'')) \le d_{G}(p,h_0^{-1}(p')) + d_G(h_0^{-1}(p'),h_0^{-1}(p'')) \le 3\d(G/e) + 1.
$$
If $p''\in [h(z)h(x)]$, then by Lemma \ref{l:dist1_G/e} there is $p_1\in[zx]$ with $d_G(h_0^{-1}(p''),p_1)\le (7/3)\,\d(G/e)$.
Then we have
$$
\begin{aligned}
\d(G)& =d_{G}\big(p,[yz]\cup[zx]\big)\le d_G(p,p_1)\\
&\le d_{G}\big(p,h_0^{-1}(p')\big) + d_G\big(h_0^{-1}(p'),h_0^{-1}(p'')\big) + d_G\big(h_0^{-1}(p''),p_1\big) \\
&\le \frac{16}3 \,\d(G/e) + 1.
\end{aligned}
$$
\end{proof}

The bounds in Theorem \ref{t:quanti} are sharp, as the following examples show.

\begin{example}\label{ex:1-0}
 Let $G_0$ be the diamond graph, \emph{i.e.}, the complete graph with $4$ vertices $K_4$ without one edge (see Figure \ref{fig:Ex1-0}).
 Let $e$ be the edge joining the two vertices with degree 3 in $G_0$. Then $G_0/e$ is isomorphic to the path graph with $3$ vertices $P_3$.
 Clearly, we have $\d(G_0)=1$ and $\d(G_0/e)=0$. This fact allows to obtain many graphs $G$ attaining the upper bound in Theorem \ref{t:quanti}:
 Consider any tree $T$ and fix vertices $v\in V(T)$ and $u\in V(G_0\setminus \{ e\})$.
 Let  $G$ be the graph obtained from $G_0$ and $T$ by identifying the vertices $u$ and $v$.
 Then $\d(G)=\d(G_0)=1$ and $\d(G/e) = \d(G_0/e)=0$, since $G/e$ and $G_0/e$ are trees.
 \end{example}

 \begin{example}
If $T$ is any tree and $e$ is any edge of $T$, then $T/e$ is also a tree,
$\d(T)=\d(T/e)=0$, and so, the lower bound in Theorem \ref{t:quanti} is attained.

 \end{example}

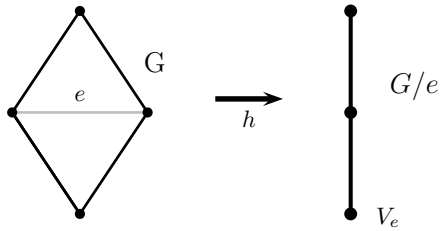
\begin{figure}[h]
  \centering
  \scalebox{0.9}
  {\begin{pspicture}(-1.2,-1)(5,3.7)
  \psline[linewidth=0.04cm,linecolor=lightgray]{-}(-1,1.5)(1,1.5)
  \psline[linewidth=0.04cm,linecolor=black]{*-*}(0,3)(-1,1.5)(0,0)
  \psline[linewidth=0.04cm,linecolor=black]{*-*}(-1,1.5)(0,0)(1,1.5)
  \psline[linewidth=0.04cm,linecolor=black]{-}(1,1.5)(0,3)
  \uput[90](0,1.5){$e$}
  \uput[40](0.8,2){\Large{G}}
  \psline[linewidth=0.06cm,fillcolor=black,dotsize=0.07055555cm 2.0]{*-*}(4,1.5)(4,3)
  \psline[linewidth=0.06cm,fillcolor=black,dotsize=0.07055555cm 2.0]{-*}(4,1.5)(4,0)
  \uput[-10](4.2,0){$V_e$}
  \uput[20](4.4,1.75){\Large{$G/e$}}
  \psline[linewidth=0.08cm,fillcolor=black]{->}(2,1.7)(3,1.7)
  \uput[-90](2.5,1.7){$h$}
  \end{pspicture}}
  \caption{An example for which the upper bound in Theorem \ref{t:quanti} is attained ($\d(G)=1$ and $\d(G/e)=0$).}
  \label{fig:Ex1-0}
\end{figure}

\smallskip

Theorems \ref{t:cutedge} and \ref{t:propercycle} below allow to improve Theorem \ref{t:quanti} in some special cases.
In order to do it we need some previous result.

\smallskip

We say that a vertex $e$ in a graph $G$ is a \emph{cut-vertex} if $G\setminus v$ is not connected.
We say that an edge $e \in E(G)$ is a \emph{cut-edge} if $G\setminus e$ is not connected.
A graph is \emph{two-connected} if it is connected and it does not contain cut-vertices.
Given a graph $G$, we say that a family of subgraphs $\{G_{s} \}_s$ of $G$ is a \emph{T-decomposition} of $G$ if $\cup_s G_{s} = G $ and $G_{s}\cap G_{r}$   is either a \emph{cut-vertex} or the empty set for each $s\neq r$. Every graph has a T-decomposition, as the following example shows. Given any edge in $G$, let us consider the maximal two-connected
subgraph containing it. We call to the set of these maximal two-connected subgraphs $\{G_s\}_s$
the \emph{canonical T-decomposition} of $G$.

\smallskip

In \cite{BRSV2} the authors obtain the following result about T-decompositions.

\begin{theorem}\cite[Theorem 3]{BRSV2}\label{t:}
Let $ G $ be a graph and $\{G_{s} \}_s$ be any T-decomposition of $G$, then
$$
\delta(G) =\sup_{s} \delta (G_{s}) .
$$
\end{theorem}

The following result improves Theorem \ref{t:quanti} when $e$ is a cut-edge.

\begin{theorem}
\label{t:cutedge}
Let $G$ be a graph and $e$ a cut-edge in $G$. Then
$$
\d(G/e) = \d(G) .
$$
\end{theorem}

\begin{proof}
Consider the connected components $\{G_{s} \}_s$ of $G\setminus e$.
Then $\{G_{s} \}_s \cup \{e\}$ is a T-decomposition of $G$ and Theorem \ref{t:} gives
$$
\delta(G)
=\max \Big\{\sup_{s} \delta (G_{s}), \, \d(e) \Big\}
=\max \Big\{\sup_{s} \delta (G_{s}), \, 0 \Big\}
= \sup_{s} \delta (G_{s}).
$$

For each $s$, let $G_{s}'$ be the subgraph of $G/e$ obtained from $G_{s}$ by replacing the vertex in $\{A,B\}$ by $V_e$.
Note that $G_{s}'$ and $G_{s}$ are isomorphic (and isometric) and, therefore, $\delta (G_{s}') = \delta (G_{s})$.
Since $\{G_{s}' \}_s$ is a T-decomposition of $G/e$, Theorem \ref{t:} gives
$$
\delta(G/e)
= \sup_{s} \delta (G_{s}')
= \sup_{s} \delta (G_{s})
= \delta(G).
$$
\end{proof}

We say that a graph \emph{has proper cycles} if each edge belongs at most to a cycle.

The \emph{circumference} $c(G)$ of the graph $G$ is the supremum of the lengths of cycles in $G$.

We need the following result.

\begin{lemma}\cite[Lemma 2.1]{RT1}\label{l:simple}
In any graph $G$,
$$
\d(G)= \sup \big\{ \d(T): \, T \,\text{ is a geodesic triangle that
is a  cycle}\,\big\} \,.
$$
\end{lemma}

\begin{proposition} \label{p:propercycle}
Let $G$ be a graph with proper cycles. Then
$$\d(G) = \frac14 c(G) .$$
\end{proposition}

\begin{proof}
In order to bound $\d(G)$, by Lemma \ref{l:simple} it suffices to consider
geodesic triangles $T = \{x, y, z\}$ that are cycles.
Hence, $L(T) \le c(G)$, $d_G(x,y)\le L(T)/2 \le c(G)/2$ and $d_G(p,[yz]\cup[zx])\le d_G(p,\{x,y\})\le d_G(x,y)/2 \le c(G)/4$
for every $p\in [xy]$.
Therefore, $\d(G) \le c(G)/4$.
Consider any fixed cycle $C$ in $G$.
Since $G$ has proper cycles, $d_C(x,y) = d_G(x,y)$ for every $x,y\in C$.
Choose $x,y \in C$ with $d_C(x,y) = L(C)/2$ and denote by $g_1,g_2$ the geodesics in $C$ joining $x$ and $y$ with $g_1 \cup g_2=C$ and $g_1 \cap g_2=\{x,y\}$.
Denote by $B$ the geodesic bigon $B=\{g_1, g_2\}$.
If $p$ is the midpoint in $g_1$, then $\d(B)\ge d_G(p,g_2) = d_G(p,\{x,y\})= d_G(x,y)/2 = L(C)/4$.
Thus $\d(G) \ge c(G)/4$, and we conclude $\d(G) = c(G)/4$.
\end{proof}

We denote by $\mathfrak{C}(G)$ the set of cycles $C$ in $G$ with $L(C)=c(G)$ if $c(G)<\infty$.
The following result improves Theorem \ref{t:quanti} for graphs with proper cycles.

\begin{theorem}
\label{t:propercycle}
Let $G$ be a graph with proper cycles.
\begin{itemize}
\item If either $c(G)=\infty$ or $\mathfrak{C}(G)$ contains at least two cycles, then $\d(G/e) = \d(G)$ for every $e\in E(G)$.
\item If $\mathfrak{C}(G)$ contains just a single cycle $C$, then $\d(G/e) = \d(G)$ if and only if $e\notin E(C)$.
\end{itemize}
\end{theorem}

\begin{proof}
Note that $G/e$ also has proper cycles for every $e\in E(G)$ and $c(G/e)$ is equal to either $c(G)$ or $c(G)-1$.
Assume first that either $c(G)=\infty$ or $\mathfrak{C}(G)$ contains at least two cycles, and fix any $e\in E(G)$.
Thus, $c(G/e) = c(G)$ and Proposition \ref{p:propercycle} gives $\d(G/e) = \d(G)$.
Assume now that $\mathfrak{C}(G)$ contains just a single cycle $C$.
If $e\notin E(C)$, then $c(G/e) = c(G)$ and $\d(G/e) = \d(G)$.
If $e\in E(C)$, then $c(G/e) = c(G)-1$ and Proposition \ref{p:propercycle} gives $\d(G/e) \neq \d(G)$.
\end{proof}

\

\section{Hyperbolicity of minor graphs} \label{Sect:minor}

In order to obtain results on hyperbolicity of minor graphs we deal now with the other transformation of graphs
involved in the definition of minor: the deletion of an edge.
Recall that $G\setminus e$ is the graph with $V(G\setminus e)=V(G)$ and $E(G\setminus e)=E(G) \setminus\{e\}$.
Theorem \ref{thm_dG_leq_6dG-e_+_Lcycle} below provides quantitative relations between $\delta(G\setminus e)$ and $\delta(G)$, where $e$ is any edge of $G$.

One can expect that the edge deletion is a monotone transformation for the hyperbolicity constant.
However, the following examples provide two families of graphs in which the hyperbolicity constant increases and decreases,
respectively, by removing some edge.

\begin{example}\label{ex:C_abc}
Consider $G_n$ as a cycle graph $C_n$ with $n\ge 3$ vertices and fix $e_n \in E(G_n)$.
Thus, $G_n\setminus e_n$ is isomorphic to a path graph $P_{n}$. Since $\delta(C_n)=n/4$ and $\delta(P_{n})=0$,
we have $\delta(G_n ) > \delta(G_n\setminus e_n)$ and
$$
\lim_{n\to \infty} \big( \delta(G_n ) - \delta(G_n\setminus e_n)\big) = \infty.
$$
\end{example}

\begin{example}\label{ex:C_abc2}
Let $C_{a,b,c}$ be the graph with three disjoint paths joining two vertices with lengths $a \le b \le c$.
Consider an edge $e$ of $C_{a,b,c}$ contained in the path with length $a$.
It is easy to check that $\delta(C_{a,b,c}\setminus e)=(c+b)/4$ and
\cite[Lemma 19]{MRSV} gives that $\delta(C_{a,b,c}) = (c+\min\{b,3a\})/4$.
If $3a< b$, then $\delta(C_{a,b,c}) < \delta(C_{a,b,c}\setminus e)$.
In particular, consider $\G_n=C_{n,4n,4n}$ and fix $e_n \in E(\G_n)$ contained in the path with length $n$.
Thus, $\delta(\G_n)=7n/4$ and $\delta(\G_n\setminus e_n)=2n$.
So, $\delta(\G_n ) < \delta(\G_n\setminus e_n)$ and
$$
\lim_{n\to \infty} \big( \delta(\G_n \setminus e_n) - \delta(\G_n)\big) = \infty.
$$
\end{example}

In \cite{CPeRS} the authors obtain quantitative information about the distortion of the hyperbolicity constant
of the graph $G \setminus e$ obtained from the graph $G$ by deleting an arbitrary edge $e$ from it.
The following theorem is a weak version of their main result.

\begin{theorem}\cite[Theorem 3.15]{CPeRS}\label{thm_dG_leq_6dG-e_+_Lcycle}
   Let $G$ be a graph and $e=[A,B]\in E(G)$ with $G\setminus e$ connected. Then
   \begin{equation}\label{ec_thm_dG_leq_6dG-e_+_Lcycle}
   \max \Big\{\frac15 \delta(G\setminus e), \frac14 \big(d_{G\setminus e}(A,B)+1\big) \Big\}
   \leq \delta(G) \leq 6 \delta(G\setminus e) + d_{G\setminus e}(A,B).
   \end{equation}
\end{theorem}

Finally, Theorems \ref{t:quanti} and \ref{thm_dG_leq_6dG-e_+_Lcycle} give the following qualitative result.

\begin{theorem}\label{t:minor}
Let $G$ be a graph and $G'$ a (connected) minor graph of $G$ obtained by a finite number of edges removed and/or contracted.
Then $G$ is hyperbolic if and only if $G'$ is hyperbolic.
\end{theorem}

\section{Non-simple graphs}\label{Sect:non-simple}

Simple graphs are the usual context in the study of hyperbolicity.
However, the operation of contraction is naturally defined for non-simple graphs.
For this reason, in this last section we study the distortion of the hyperbolicity constant by contraction of one edge in non-simple graphs.

\smallskip

Since we work with non-simple graphs, if there are $n_1\ge 1$ edges in $G$ joining $v$ and $A$
and $n_2\ge 1$ edges joining $v$ and $B$ for some $v\in V(G)$, then we obtain $n_1+n_2$ edges joining $v$ and $V_e$ in $G/e$,
see Figure \ref{fig:OutlineH2}.
Thus, in the context of non-simple graphs a cycle
$C \in \mathcal{C}(G,e)$ is transformed in a double edge in $G/e$,
as in Figure \ref{fig:OutlineH2}.
Furthermore, for each edge $e'\neq e$ joining $A$ and $B$ in $G$, we obtain a loop at $V_e$ in $G/e$.

\smallskip

We define the map $H: G \rightarrow G/e$ in the following way:
if $x$ belongs to the edge $e$, then $H(x):=V_e$;
if $x\in G$ does not belong to $e$, then $H(x)$ is the ``natural inclusion map".
Clearly $H$ is onto, \emph{i.e.}, $H(G)=G/e$.
Besides, $H$ is an injective map in $G\setminus\{e\}$.

\begin{figure}[h]
  \centering
  \scalebox{1}
  {\begin{pspicture}(-1.2,1)(5,3.5)
  \cnode*(-1,1.5){0.1}{A1}
  \cnode*(1,1.5){0.1}{A2}
  \cnode*(0,3){0.1}{A3}
  \ncarc[arcangle=0,linewidth=0.06cm,linecolor=lightgray]{-}{A1}{A2}
  \ncarc[arcangle=0,linewidth=0.04cm]{-}{A1}{A3}
  \ncarc[arcangle=0,linewidth=0.04cm]{-}{A3}{A2}
  \ncarc[arcangle=45,linewidth=0.04cm]{-}{A2}{A1}
  \uput[90](0,1.5){$e$}
  \uput[40](0.6,2.6){\Large{G}}
  \cnode*[linecolor=gray](4,1.5){0.1}{A4}
  \cnode*(4,3){0.1}{A5}
  \ncarc[arcangle=20,linewidth=0.04cm]{-}{A4}{A5}
  \ncarc[arcangle=-20,linewidth=0.04cm]{-}{A4}{A5}
  \nccircle[angleA=180,linewidth=0.04cm]{-}{A4}{.25}
  \uput[-10](4.2,1.5){$V_e$}
  \uput[20](4.4,2.6){\Large{$G/e$}}
  \psline[linewidth=0.08cm,fillcolor=black]{->}(2,2.2)(3,2.2)
  \uput[-90](2.5,2.2){$H$}
  \end{pspicture}}
  \caption{The map $H$.}
  \label{fig:OutlineH2}
\end{figure}
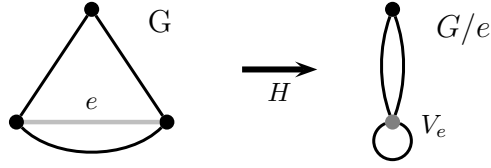

We prove now a version of Proposition \ref{l:quasi_isom} for non-simple graphs.

\begin{proposition}\label{l:quasi_isom2}
  Let $G$ be a non-simple graph and $e\in E(G)$. Then
   \begin{equation}\label{eq:+12}
     d_{G/e}\big(H(x),H(y)\big) \le d_{G}(x,y) \le d_{G/e}\big(H(x),H(y)\big) + 1,
   \end{equation}
for every $x,y\in G$.
\end{proposition}

\begin{proof}
Fix $x,y \in G$. Let $[xy]_G$ be a geodesic in $G$ joining $x$ and $y$.
Clearly, $H([xy]_G)$ is a path joining $H(x)$ and $H(y)$ with length at most $L([xy]_G)$.
Hence, we obtain $d_{G/e}\big(H(x),H(y)\big) \le d_{G}(x,y)$.

Let $\g'$ be a geodesic in $G/e$ joining $H(x)$ and $H(y)$.
Then there is a path $\g$ in $G$ with $H(\g)=\g'$ and $L(\g)\le L(\g')+1$ since
$\g$ can contain $e$ or a subset of $e$.
Therefore, $d_{G}(x,y) \le L(\g)\le L(\g')+1 = d_{G/e}\big(H(x),H(y)\big) + 1$.
\end{proof}

Note that the inequalities in \eqref{eq:+12} are attained.
If $G$ is any non-simple graph, $\{v,w\}\neq \{A,B\}$ and $[v,w]\in E(G)$, then $d_{G/e}(H(v),H(w))=1=d_G(v,w)$.
If $G$ is any non-simple graph, then $d_G(A,B)=1=d_{G/e}(H(A),H(B))+1$.

\smallskip

For any simple path $\g'$ joining two different points in $G/e$, $H^{-1}(\g')$ is either a simple path $\g$ in $G$
or the union of a simple path $\g$ with $e$.
In both cases, we define $H_0^{-1}(\g'):=\g$.
If $\a'\in \g'$, then we define $H_0^{-1}(\a'):= H^{-1}(\a') \cap H_0^{-1}(\g')$
(in particular, $H_0^{-1}(V_e)$ can be either $A$, $B$ or $e$).
Furthermore, if $V_e\notin \g'$, then $H_0^{-1}(\g')= H^{-1}(\g')$ and
$$
H\!\mid_{H^{-1}(\g')} : H^{-1}(\g') \longrightarrow \g'
$$
is a bijective map.

\smallskip

One can check that the following simpler versions for non-simple graphs of Lemmas \ref{l:dist2_G}, \ref{l:dist1_G/e} and \ref{l:x} hold.

\begin{lemma}\label{l:dist2_G2}
   Let $G$ be a non-simple graph and $e \in E(G)$.
   Assume that for some $x,y\in G\setminus\{e\}$ there are two geodesics $\g_G$ and $\g_{G/e}$ in $G$ and $G/e$,
   respectively, joining $x,y$ and $H(x),H(y)$, respectively, such that $L(\g_G)=L(\g_{G/e})=L(H(\g_G))$ and $e\subset H_0^{-1}( \g_{G/e} )$. Then we have
   $$
     d_{G/e}(H(\a),\g_{G/e} )\le \d(G) \quad \forall\ \a\in \g_{G}
   $$
   and
   $$
     d_{G/e}\big( \a',H(\g_{G}) \big)\le 2 \d(G) \quad \forall\ \a'\in \g_{G/e}.
   $$
\end{lemma}

\begin{lemma}\label{l:dist1_G/e2}
   Let $G$ be a non-simple graph and $e\in E(G)$ such that $G/e$ is not a tree.
   Let $[xy]_G$ be a geodesic in $G$ joining $x,y \in J(G)$.
   Assume that $H([xy]_G)$ is not a geodesic in $G/e$ and let $[H(x)H(y)]_{G/e}$ be a geodesic in $G/e$ joining $H(x)$ and $H(y)$.
   Then we have
   $$
     d_{G}(H_0^{-1}(\a'),[xy]_G)\le \d(G/e)+1 \le \frac73 \,\d(G/e), \quad \forall\ \a'\in [H(x)H(y)]_{G/e}
   $$
   and
   $$
     d_{G}\big(\a,H_0^{-1}([H(x)H(y)]_{G/e}) \big)\le 2 \d(G/e), \quad \forall\ \a\in [xy]_{G}.
   $$
\end{lemma}

\begin{lemma} \label{l:x2}
Let $G$ be a non-simple graph and $e\in E(G)$. Then $G$ is a tree if and only if $G/e$ is a tree.
\end{lemma}

Hence, the main results also hold for non-simple graphs.

\begin{theorem}
\label{t:quanti2}
Let $G$ be a non-simple graph and $e\in E(G)$. Then
$$
  \frac13\,\d(G/e) \le \d(G) \le \frac{16}3 \,\d(G/e) + 1.
$$
\end{theorem}

\begin{theorem}
\label{t:cutedge2}
Let $G$ be a non-simple graph and $e$ a cut-edge in $G$. Then
$$
\d(G/e) = \d(G) .
$$
\end{theorem}

\begin{theorem}\label{t:minor2}
Let $G$ be a non-simple graph and $G'$ a (connected) minor graph of $G$ obtained by a finite number of edges removed and/or contracted.
Then $G$ is hyperbolic if and only if $G'$ is hyperbolic.
\end{theorem}

%
%
%

\

\end{document}